\documentclass[12pt,a4,twoside]{amsart}
\usepackage[pdftex]{graphicx}

\usepackage[T2A]{fontenc}
\usepackage[utf8]{inputenc}
\usepackage[english]{babel}

\usepackage{amsmath}
\usepackage{amssymb}
\usepackage{amscd}
\usepackage{amsfonts,amsmath,amsthm}

\usepackage[matrix,arrow,curve]{xy}
\usepackage{xfrac}
\usepackage{color}
\usepackage{url}
\usepackage{hyperref}
\hypersetup{colorlinks=true,linkcolor=blue,urlcolor=blue}

\DeclareMathOperator{\sq}{\mathbb{D}}

\DeclareMathOperator{\reg}{reg} 
                
\DeclareMathOperator{\sing}{sing}
\DeclareMathOperator{\atyp}{atyp}
\DeclareMathOperator{\iflex}{iflex}
\DeclareMathOperator{\focal}{focal}

\DeclareMathOperator{\ED}{ED}
\DeclareMathOperator{\tED}{\ED*}

\DeclareMathOperator{\hot}{hot}

\newtheorem{Ex}{Example}
\newtheorem{theorem}{Theorem}
\newtheorem{prop}{Proposition}
\newtheorem{cor}{Corollary}
\newtheorem{lemma}{Lemma}
\newtheorem*{main lemma}{The main lemma}
\newtheorem*{remark}{Remark}
\newtheorem{definition}{Definition}

\title[Normals and Double Normals for plane Curves  ]{Distance, Normals and Double Normals for real  plane Curves with Singularities}

\author{Dirk Siersma}
\address{D. Siersma, Mathematisch Instituut, Universiteit Utrecht, P.O.Box 80.010 , 3508TA Utrecht, The Netherlands, d.siersma@uu.nl}
\date{\today}

\thanks{The author acknowledges support from the project ``Singularities and Applications'' - CF 132/31.07.2023 funded by the European Union - NextGenerationEU - through Romania's National Recovery and Resilience Plan, and support by the grant CNRS-INSMI-IEA-329. }

\begin{document}

\begin{abstract}
For real algebraic curves in the plane with singularities we investigate the relation between normals and double normals and the critical points of the squared distance function (up to topological equivalence). For the distance to a given point we show that the topological discriminant consists of the (traditional) evolute, together with some distinguished normal lines at algebraic singular points. We pay special attention to the behavior at points, where the curve is smooth, but only $C^1$ embedded. We discus the differences and similarities with the ED-discriminant in the complex theory. We give counting formula's for normals and double normals, relating them with maxima and minima of the distance function.
\end{abstract}

\maketitle

\section{Introduction}

The study of algebraic curves in the plane and their singularities has a long history.  In the complex setting  a very detailed description is given by the Puisseux expansion. For the real setting we refer especially to Bureau \cite{Bu}, Brieskorn-Kn{\"o}rrer \cite{BK} and \cite{PRS}. The caustic (or evolute) of a smooth plane curve as envelope of the normals is a classical subject in differential geometry. It is related to focal points and to the singularities of the offsets. The singularities of the caustic occur at vertices,  the points where the derivative of the curvature vanishes. 
Normals correspond to critical points of the squared distance function from a given point to the curve. Generically they come with a Morse-index and this  enables us to do Morse-counting of normals. Similar ideas apply to double normals. In this paper we include singularities and describe  new phenomena.

We consider reduced real algebraic curves $X$ given by $\Phi(x_1,x_2)=0$.  \textit{Algebraic singular points} are the points of  $X$  where the gradient vanishes: $\nabla \Phi = 0$. Algebraic singular points are isolated.  
 At each singular point $X$
 is the union of finitely many local branches. We will assume that they are all real and 1-dimensional; so we omit void components  and points  (such as $x_1^2 + x_2^2 = c$ with $c \le 0$).  Branches can be smooth or singular; but each branch has a well defined tangent line at each point of $X$ and therefore also a well defined normal line.

One can also describe real branches as images of  parametrizations $(x_1(t),x_2(t))$.
These are maps $\mathbb{R}\to \mathbb{R}^2 $ or $S^1 \to \mathbb{R}^2$. In case of self intersections these  mappings serve as normalizations of $X$.

In contrast to  the complex case, having an algebraic singularity can still imply that a local branch is smooth, more precisely  a $C^k$ immersed submanifold ($k \ge 1$) of $\mathbb{R}^2$. See the Proposition \ref{p:k-smooth}. The remaining algebraic critical points are called \textit{geometric critical points}, or \textit{cuspidal} due to their shape.

Next we introduce the squared Euclidean Distance (ED) function
\begin{equation}\label{e:sqd0}
\sq_u:X\rightarrow \mathbb{R}, \ \ \  \sq_u(x):=|x-u|^2.
\end{equation}
and its extension the the normalization of $X$. We will consider topological critical points. These can only be minima or maxima; the other points are topologically regular.
In Proposition \ref{p:oddcrpt}   and \ref{l:outofnormal}  we show that for smooth points of the curve $\sq_u$ is critical iff $u$ belongs to the normal line. For cuspidal points $\sq_u$ is generically topologically singular. We will call the corresponding interval $[ux]$ a s-normal from $u$.

Next we define the topological ED-discriminant $\Delta_{tED}$ as the complement of the set of stable points $u \in \mathbb{R}^2$, stability with respect to the local topological type. In Theorem \ref{t:top-discr} we show that
$\Delta_{tED}$  consists of the caustic (evolute), together with all normal  lines  at  geometric singular points  and  sometimes the normal lines at strict $C^1$-points (smooth points, where $X$ is not a $C^2$-embedding):
$$  \mbox{\rm tED-discriminant}  \; = \mbox{\rm evolute}  \; \cup \;  \mbox{\rm distinguished  normals}$$
We  present an important local  ingredient  in  Theorem \ref{t:NinD}:  we test if a normal is included in the discriminant and determine the bifurcation of critical points and the focal sets.  The proofs are rather forward computations.

We introduce the concept of \textit{focal point} in our more general context of singular points  and determine its location on the normal, depending on a \textit{competition factor} and determine when the focal point coincides with $x \in X$. 
The limits of focal sets at  singular points coincide with the focal point on the normal line, the limit is infinite if $\rho > 1$, finite if $\rho=1$ and coincides with $x \in X$ if $\rho < 1$. See Theorem \ref{t:NinD}, where we also give a full description of discriminant and bifurcation inTable  \ref{table:10cases} and Figure  \ref{fig:TenCases}.

 Near algebraic critical points it can happen, that $X$ is smooth, but  only $C^1$.  This case  is interesting and differs from the smooth  $C^2$-case; e.g. the normal line can belong to $\Delta_{tED}$ and there can occur some special bifurcations of critical points.  Also standard Morse theory does not applies in this case. 

In Proposition 
\ref{p:count-normals} we present the counting of normals (including s-normals, related to singular points) and show that 
each cuspidal point contributes to at least one normal:
$$ \sharp (normals) \ge \sharp (cusps) $$
Section \ref{s:double} deals with double normals.  We first study the squared distance between two curves, which both may have singularities
$$\mathbb{B}  : X_a \times X_b \rightarrow  \mathbb{R}, \ \ \ \mathbb{B}(a,b)  =|a-b|^2.$$
and determine its critical points in Proposition \ref{p:double-crit}. The critical points correspond to usual double normals on the smooth part, together with rs-normals and ss-normals (related to singular points) and intersection points. In the smooth case we compare the Morse type of double normals with the local Morse type at the end-points.

Finally we apply this to the (classical) case of \textit{double normals} on a single curve and state counting results in Propositions \ref{p:double-count} and \ref{p:extra-selfdouble}: The number of double normals is at least 2, every cuspidal point contributes to at least one double normal, Every pair of cuspidal points contributes to a double normal. Self-intersections contribute  as
 minimima.
$$ \sharp(double \; normals) \ge \sharp (pairs \; of \; cusps)$$
The complex case  has been studied in \cite{JST}. The (generic) number of normals on the regular part of the variety is known as the ED-degree,  excluding the normals at singularities.
The real case is different: it is important to include singular points  in order to get maximal or minimal distances.\\

\section{Normals and the squared distance function in the presence of singularities}

\subsection{Types of algebraic singular points}
We assume, with respect to the tangent line,  a local parametrization:
\begin{equation}\label{e:Puis}
x_1(t) = t ^p \; , \;  x_2(t) = c \,t^q + \sigma_{q+1}(t) 
\end{equation}
with $q >p \ge 2$, deg $\sigma_{q+1} > q$ and $c \ne 0$.
We can even assume $c >0$ .

\begin{prop}\label{p:k-smooth}
Let $p = odd$   and $k = [q/p]$, then a branch of  $X$ is locally an embedded $C^k$ manifold.

\end{prop}
\begin{proof}
We take a new parametrization with $t^p=s$ and get:
\begin{equation}
 x_1(s) = s \; , \;  x_2(s) =c \, s^{q/p} + \sigma_{q+1}(s^{1/q}) =c\,s^k . \tau (s^{1/p}) ,
 \end{equation} \label{e:puis-par}
where $k = [q/p]$.
This is clearly a $C^{k}$ local immersion.
\end{proof}  

\begin{definition} 
For a local branch we distinguish (see also Figure \ref{fig:fourcusps}):
\begin{itemize}
\item \textit{\bf smooth points}, when the local branch is of type $p = odd$, i.e the curve is an embedded $C^1$-manifold, The smooth point is called a {\bf strict $C^1$-point} if the curve is not a $C^2$ imbedded manifold, i.e  when $p = odd$ and $q  <2p$.
\item  \textit{\bf geometric singular points} or \textit{\bf cuspidal points}  $x \in X$, where there is a  local branch is of type $p = even$, i.e  the local branch is not a $C^k$ immersion ($k \ge 1$).
\end{itemize}
\end{definition}

\vspace{-0.6cm}
\begin{center}
\begin{figure}[h]
\includegraphics[width=14cm]{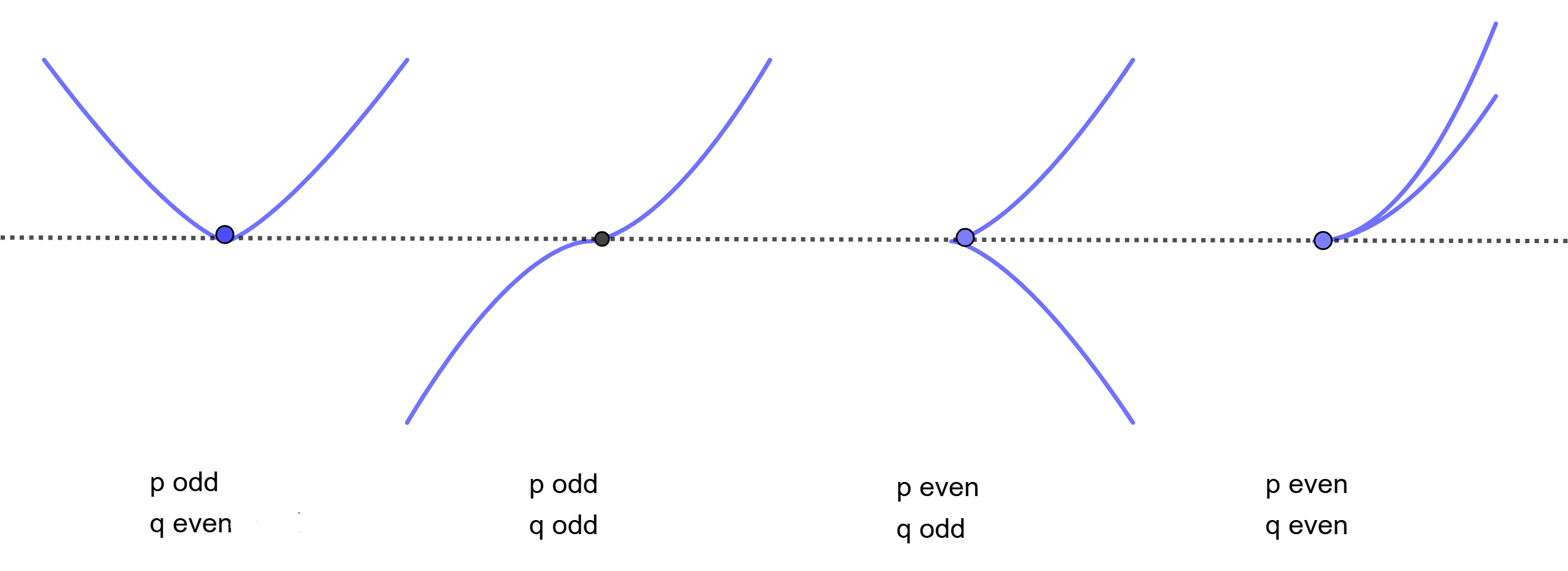}
\caption{Four types of branches. The first and second are (at least) $C^1$-smooth. The  singularity of the $3^{rd}$ curve is called cusp  of the first kind; on the $4^{th}$ curve: cusp of second kind.}\label{fig:fourcusps}
\end{figure}
\end{center}
\vspace{-0.6cm}
Algebraic critical points also occur as  multiple intersections of local branches, smooth or singular.
Tangent lines are defined at every point of $X$ for each branch smooth or singular.
We can define normal lines for all $x \in X$: the line $N_{\Gamma,x}$, which is orthogonal to the tangent line at a local branch $\Gamma$.
We write $N_x$ for the union of those normal lines at $x$ (including the case of a single branch).

\subsection{Squared distance function}

The squared distance function in presence of singularities also has been studied in  \cite{De} in the general context of  sets definable in  polynomially bounded o-minimal structures. The main focus in that paper were on properties of the medial  axis. The cases, where $X$ is algebraic and $C^1$ but not $C^2$ have been studied under the name \textit{supersquared}. \\

Our parametrization  is not  a regular parametrization near singular points. We don't use the differential structure of $X$ but of its normalization. Up to homeomorphisms it gives a  topological description of $\sq_u$. 
$\sq_u$ is called topologically regular \textit{top-regular} at $x \in X$ when is locally homeomorphic to a linear function.
It follows that (besides the constant function) the only possibilities for the local behavior are: top-regular (increasing or decreasing), minimum or maximum. The following lemma is evident.

\begin{lemma}
On each (normalized) branch 
the topological type is completely determined by the number of maxima and mininma.  

\end{lemma}

\begin{definition}
A {\bf focal point} on a normal is a point $u$, where the type of $\sq_u$ is different from the generic types on the normal. In other words: the center of the best-tangent circle.   In the smooth $C^2$-case this generic type  is Morse in the usual parametrization.

\noindent
The {\bf Caustic} is the collection of focal points on the normal  lines  (including those at geometric singular points). It extends the definition of evolute to the algebraic singular points.
\end{definition}

We recall, that in the smooth $C^2$-case there are 3 equivalent definitions of caustics:
\begin{itemize}
\item centers of curvature (best tangent circles),
\item envelope of normal lines,
\item singularities of the equidistant curves.
\end{itemize}

Note that if $p=odd$ and $q \ge 2p$ there is  $C^2$-differentiability and the theory of focal points  is covered by the usual Morse theoretic approach to differential geometry \cite{Milnor}. This includes Milnor's magic formula, which in the curve case states:
{\em The index is 0 if there is no focal point between u and x and the index is 1  if the the focal point is outside the closed interval ux. }

We say that $u$ is a \textit{topologically stable point for $\sq_u$} if there is a neighborhood $U$ of $u$ such that for all $v \in U$ the
functions $\sq_v$ are topologically equivalent to $\sq_u$. We define the {\bf tED-discriminant} $\Delta_{tED}$ as complement of the stable points.

\subsection{Critical points at the algebraic singular points of X}
We first consider real smooth points  ($p = odd$). In this case the concept of critical point is (as usual) defined by the vanishing of the gradient. For $C^1$-functions holds:
If the gradient is non-zero in $x \in X$  then the function is locally topologically ($C^0$) regular. A reference is \cite{CE}; 
it is a by-product of a very general theorem on page 38. 

\begin{prop}\label{p:oddcrpt} Let  $x \in X$ be a  smooth point then\\
$ \sq_u$ is critical $\Longleftrightarrow$  $u$ belongs to the normal $N_x$.
\end{prop}
\begin{proof}
This follows from 
\begin{equation}\label{e:sqd}
\sq_u(t)  = u_1^2 + u_2^2 - 2u_1 t^p + t^{2p}  - 2 c u_2 t^q    -2 u_2 \sigma_{q+1}(t)   + c^2 t^{2q} + \hot  
\end{equation}

\end{proof}

Next we consider  $p = even$: the geometric  singular points.
\begin{definition}
The function $\sq_u$ is called critical at a  geometric singular point $x \in X$ when it is not topologically regular at $x \in X$.

\end{definition}

\begin{prop}\label{l:outofnormal}
Let $x \in X$ be a geometric singular point.
For all points outside the normal  line $N_x = \{u_1=0\}$ the function $\sq_u$ is critical
 in $x$; we have the following types of critical points:
\begin{itemize}
\item $u_1 > 0$ maximum,
\item $u_1 < 0$ minimum.
\end{itemize}
For points on  $N_x$ see Proposition \ref{p:focal}.
\end{prop}
\begin{proof}
Look at formula (\ref{e:sqd})
\end{proof}

\subsection{Focal points on normal lines}
We define the   \textit{\bf competition factor}
 $\rho = \frac{q}{2p}$. It describes  the contact between the squared distance function and the curve at $x$. For generic smooth points the factor is $1$.

 \begin{prop}\label{p:focal}
  Let $x \in X$ be an algebraic singular point and $u \in N_x$, then 
  \begin{itemize}
\item If  $\rho > 1,  \; (q > 2p)$  then $\sq_u$ is a minimum for all $u \in N_x$, there is no focal point on $N_x$,
\item If $\rho = 1, \; (q = 2p)$ then there is a focal point $u \ne x$, where a minimum changes into  a maximum,
\item If $\rho < 1, \;(q < 2p)$ and $u_2 \ne 0$ then $\sq_u$ is top-regular if $q=odd$ and maximum or minimum if $q= even$, $u=x$ is the  single focal point.
\end{itemize}
If $u_1=u_2=0$ then $\sq_u$ has a minimum.

  \end{prop}
\begin{proof}
Next we consider the normal line $u_1$=0;  we get from (\ref{e:sqd}):
\begin{equation}\label{e:normaleq}
 \sq_u(t)  =  u_2^2 + t^{2p}  - 2 c u_2 t^q    -2 u_2 \sigma_{q+1}(t)   +c^2  t^{2q} + hot  .
 \end{equation}
 First consider 
 \fbox{$q < 2p$}: the  term $-2 c u_2 t^q$  dominates and the result now depends on $q$. \\
  If $q = odd$ then we have a topological regular situation if $c u_2 \ne 0$. For $u_2=0$:  we have a minimum: The only focal point is $u=x$.\\
  If $q = even$ then we have a maximum if $c u_2 >0$; a minimum if $c u_2\le  0$. The only focal point is $u=x$.\\
\fbox{$q > 2p$}: the term $t^{2p}$ is dominant for all $u_2$. The type is minimum for all points on the normal. No focal points.\\
\fbox{ $q =2p$}:
 The dominant term in (\ref{e:normaleq})  is $(1-2 c u_2) t^{2p}$. This means if $2 c u_2 < 1$,  then a minimum; if $2c  u_2 > 1$, then  a maximum;  and if $2 c u_2 =1$ a focal point at $(0,\frac{1}{2c})$.
 So on this normal we have a finite focal point, which is center of curvature (also at geometric singular points !).

 So this situation looks very much the same as in the smooth $C^2$-case:
  If $q \ge 2p$ the the conclusion follows from the $C^2$ statements.
  \end{proof}

\begin{theorem}\label{t:NinD} 
Let $x \in X$ be an algebraic singular point.
The normal  line $N_x$ is contained in  the discriminant if
\begin{itemize}
\item  $x$ has $p$ and $q$  odd  with $\rho < 1$ , or
\item $x$ is a geometric singular point.
\end{itemize}
The limit of the focal set is\begin{itemize}
\item the point $u=x$ if $\rho < 1$, 
\item the finite focal point $u \ne x$ on $N_x$ if $\rho =1$ and 
\item the infinite point on $N_x$ if $\rho > 1$.
\end{itemize}

This is summarized in Table \ref{table:10cases}, where we also give details about bifurcations of critical points.
\end{theorem}

\begin{table}[ht]
\centering
\begin{tabular}{|l|c|c|c|c|}
\hline
    & $\rho < 1$      &  $\rho = 1$    &  $\rho > 1$  \\
\hline
 $p = odd$ ; $q= odd$       & $q < 2p$ & $q=2p$ & $ q> 2p$\\
\hline
$N \subset$ Discriminant &  YES  & &  NO  \\
Type-change & $M + m \leftrightarrow \emptyset$ &  &  $m$ or $M$ persists\\
Focal point    & origin &   & no\\
Focal set limit & origin &  & infinity \\
\hline
 $p = odd$ ; $q= even$       & $q < 2p$  &   $q=2p$   & $ q> 2p$\\
\hline
$N \subset$ Discriminant & NO & NO  & NO  \\
Type-change & $m$ or $M$ persists &  $m$ or $M$ persists  &   $m$ or $M$ persists\\
Focal point    & origin &  point on $N$  & no\\
Focal set limit & origin & focal point & infinity \\
\hline
 $p = even$; $q =odd$       & $q < 2p$  &   $q=2p$ & $ q> 2p$\\
\hline
$N \subset$ Discriminant &  YES  &  & YES \\
Type-change & $M + m \leftrightarrow M+m$ &   &   $m \leftrightarrow 2m+M $ \\
Focal point    & origin &   & no\\
Focal set limit & origin &  & infinity \\
\hline
 $p = even$ ; $q= even $      & $q < 2p$   &   $q=2p$   & $ q> 2p$\\
\hline
$N \subset$  Discriminant & YES & YES & YES\\
Type-change & $2M + m \leftrightarrow M$ &  $2M + m \leftrightarrow M $   &   $2M + m \leftrightarrow M $ \\
Focal point    & origin & point on $N$ & no\\
Focal set limit & origin & focal point & infinity \\
\hline
\end{tabular}
\caption{The 10 cases; the table refers to $c > 0$. Depending on the sign of $u_2$ sometimes $m$ (minimum) or $M($ maximum) has to be interchanged.}
\label{table:10cases}
\end{table}

\begin{proof}
The starting point in the proof is the expression:
\begin{equation}\label{e:sqd2}
\sq_u(t)  = u_1^2 + u_2^2 - 2u_1 t^p + t^{2p}  - 2 c u_2 t^q    -2 u_2 \sigma_{q+1}(t)   + c^2 t^{2q} + \hot  
\end{equation}

For sake of exposition we will assume, that the higher order terms are zero. These terms cause higher order effects in the algebraic computation, but no qualitative expects for the local topology and geometry. We will comment on the case $p = even$ $q = even$ in Remark  \ref{r:even-even}.

 Critical points of $\sq_u$  satisfy $\frac{d}{dt}(\sq_u )= 0$ and therefore the following 
orthogonality condition  between the tangent line and Euclidean distance $\sq_u$ holds:
\begin{equation} \label{e:normality-0}
(u_1 -t ^p) p t^{p-1} + (u_2 - c t^q) q c t ^{q-1}= 0.
\end{equation}

Expanding $\frac{d}{dt}(\sq_u )$  gives
\begin{equation}  \label{e:calculus}
  - 2 t^{p-1}  (p  u_1 -p t ^{p}   + q c u_2 t ^{q-p} - q c ^2 t ^{2q-p}).
\end{equation}
This has anyhow a root of multiplicity $p-1$ at $t=0$.

Standard sign rules at the zero's of (\ref{e:calculus}) determine the type of extrema of $\sq_u$ and their bifurcations when we fix $u_2 \ne 0$ and  let $u_1$ move.
Taking limits $u_1 \to 0$ will involve the dominant terms in the second factor. 

Consider three cases, depending on the value of $\rho$: \\
\fbox{$\rho <1$} The dominant terms in  the derivative (\ref{e:calculus})  are:
\begin{equation}
- 2 t^{p-1}  (p  u_1  + q c u_2 t ^{q-p} )
\end{equation}
\fbox{$p = odd$ ; $q = odd$} Roots $t \ne 0 $ are given by $t^{even} = - pu_1 /  qcu_2$ . If $u_1cu_2 < 0$ then there are two roots $t_{-}, t_{+}$ with different signs. If $u_1cu_2 >  0$  no roots. At $t=0$  $\sq_u$ is top-regular. As a result there is  birth or death of a maximum (M) and minimum (m). Therefore $N_x$ belongs to the discriminant.\\
\fbox{$p = odd$ ; $q = even$}  Roots $t \ne 0 $ are given by $t^{odd} = - pu_1 /  qcu_2$ . If $u_1cu_2 \ne  0$ then there is always one root $\hat{t}$, which is positive or negative, depending on the sign of $c u_2u_1$. At $t=0$  maximum or minimum. Conclusion: Maximum or minimum survives. Therefore $N_x$ does not belong to the discriminant.\\
\fbox{$p = even$ ; $q = odd$.} The roots are as in $p = odd$ ; $q = even$. The difference is that $t=0$ is now a minimum or a maximum. This induces $M+m \leftrightarrow M+m$ for $u_1 \ne 0$, but at $u_1=0$ we have top-regular; therefore $N_x$ belongs to the discriminant, but there is no type change in the complement!\\
\fbox{$p = even$ ; $q = even$.} The roots are as in $p = odd$ ; $q = odd$. The difference is that $t=0$ is now a minimum or a maximum. This induces $2M+m \leftrightarrow M$ (or  $m+2M  \leftrightarrow M$); therefore $N_x$ belongs to the discriminant.\\
\fbox{$\rho >1$} The dominant terms in  the derivative (\ref{e:calculus})  are:
\begin{equation}
- 2 t^{p-1}  (p  u_1  - p t ^{p} )
\end{equation}
\fbox{$p = odd$} Roots $t \ne 0 $ are given by $t^{p} = u_1$.
If $u_1\ne  0$ then there is always one root $\hat{t}$, which is positive or negative, depending on the sign of $u_1$; but a minimum in both cases. Also $t=0$ is a minimum. Conclusion: minimum survives. Therefore $N_x$ does not belong to the discriminant.\\
\fbox{$p = even$} Roots $t \ne 0 $ are given by $t^{p} = u_1$. If $u_1 > 0$ there are 2 roots with different signs, which give both a minimum. In between $t=0 $: a maximum. If $u_1 < 0$ there is a single minimum 
 This induces $m \leftrightarrow 2m+M$ for $u_1 \ne 0$, therefore $N_x$ belongs to the discriminant.\\
\noindent
\fbox{$\rho =1$} The dominant terms in the derivative (\ref{e:calculus})  are:
\begin{equation}
- 2 t^{p-1}  (p  u_1  - t ^{p}  + q c u_2 t^{q-p}) = - 2 t^{p-1}  (p  u_1  - (1   - 2p c u_2 ) t^{p})
\end{equation}
\fbox{$p = odd$}. The root at $t=0$ has even multiplicity: so top-regular. Roots from the other factor satisfy $(1-2 cu_2) t^{odd} = u_1$.  There will be exactly 1 critical point $\hat{t}>0$ or $\hat{t}< 0$, depending on the sign of $u_1 (1 -2cu_2)$. For all $t$ near $0$ there is one critical point of constant type (maximum or minimum depending on signs of $u_1,u_2,c$).
Conclusion: $N_x$ does not belong to the discriminant.\\
\fbox{$p=$even}. The root at $t=0$ has odd multiplicity: so maximum or minimum. Roots from the other factor satisfy $(1-2 cu_2) t^{even} = u_1$. 
Depending on the sign of $u_1 (1 -2cu_2)$ there will be exactly 2 critical points $\hat{t}>0$ and $\hat{t}< 0$ of the same type   or no critical points. The type changes see table.  
Conclusion: $N_x$  belongs to the discriminant.\\

\noindent
Next we turn to the limit of the focal set, when $t \to 0$.
After dividing  (\ref{e:calculus}) by $t^{p-1}$ we get the equation:
\begin{equation}\label{e:calculus-red}
 p u_1 -  p t^p + q  c u_2 t ^{q-p} -  q c^2  t ^{2q-p} = 0.
\end{equation}
To get focal points we differentiate to $t$:
\begin{equation}\label{e:focals-red}
- p^2 t^{p-1} + q (q-p) c  u_2t^{q-p-1} - q (2q-p) c^2 t^{2q -p-1} = 0.
\end{equation}

\fbox{case $q < 2p$:}  Divide  (\ref{e:calculus-red}) by $t^{q-p-1}$: 
\begin{equation}\label{e:focals-small} 
-p^2 t^{2p-q} + q (q-p)c  u_2 - q (2q-p) c^2 t^q = 0.
\end{equation}
In this case the closure of the caustic (evolute) contains the origin. 
This also follows (in a more general context) from theorem 3.13 in \cite{BD}.

After re-arranging formula's (\ref{e:calculus-red}) and (\ref{e:focals-small}) we get a parametrization:
\begin{equation}\label{e:u1}
p (q-p) u_1 = p(q-2p) t^p - q^2 t^{2q-p},
\end{equation}
\begin{equation}\label{e:u2}
q (q-p) c u_2 = p^2 t^{2p-q} + q (2q-p)  c^2 t^q,
\end{equation}
where the exponents are in increasing order:
 $2p-q < p < q < 2q-p$. \\
This shows that (\ref{e:u1}),(\ref{e:u2}) satisfy an equation $\Psi (u_1,u_2) = 0$ , where the principal monomials are $u_1^{2p-q}, u_2^{p}$. So we get a cusp of type $(2p-q, p)$ from a cusp of type $(p,q)$.\\
Note that the caustic is smooth when $2p = q+1$.\\
\fbox{case $q > 2p$:} Divide (\ref{e:calculus-red})  by $t^{p-1}$: 
\begin{equation}\label{e:focals-big}
 -p^2 + q (q-p) c  u_2 t^{q -2p}- q (2q-p) c^2 t^{2q-2p }= 0.
\end{equation}
The equations (\ref{e:calculus-red}) and (\ref{e:focals-big})  describe the closure of the caustic. This closure does not contain the origin. After re-arranging these formula's we get:
\begin{equation}\label{e:u1-big}
p (q-p) u_1 = p(q-2p) t^p - q^2 t^{2q-p},
\end{equation}
\begin{equation}\label{e:u2-big}
q (q-p) c u_2 = p^2 \frac{1}{ t^{q-2p}} + q (2q-p)  c^2 t^q.
\end{equation}
This shows that if $t \to 0$ then the limit is infinity, asymptotic to the normal line $u_1=0$.

\begin{figure}[b]
\includegraphics[width=13.5cm]{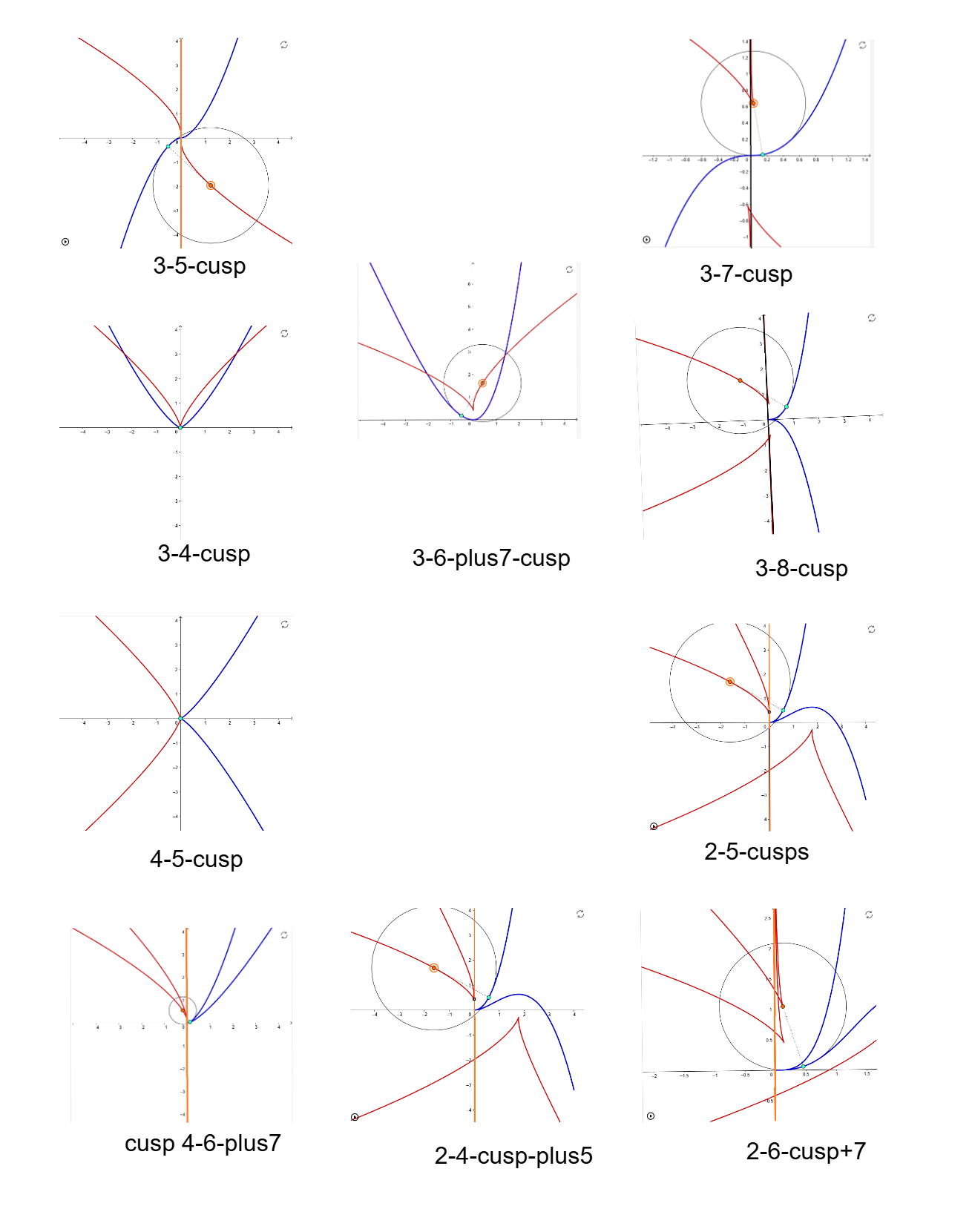}
\caption{Ten types of discriminants.\label{fig:TenCases}\\
\small{The composition of pictures corresponds with Table \ref{table:10cases}; the notation `plus7' means that an extra term $t^7$ is added, etc.}}
\end{figure}

\fbox{Case $q = 2p$:}  Divide (\ref{e:calculus-red})  by $t^{p-1}$:
and  equation (\ref{e:u2-big}) becomes
\begin{equation}
- p^2 + 2p^2 c  u_2 = 0 \; \; ;  \;  2 u_2 = 1.
\end{equation}
There is a single focal point  $(0,\frac{1}{2})$.
Formula's  (\ref{e:u1-big}),(\ref{e:u2-big}) define a $(2,3) $ cusp:
\begin{equation}
p ^2 u_1 =  - 4p^2 t^{3p} \;  
\end{equation}
\begin{equation}
2 p^2 c u_2 = p^2 + 6 p^2 t^{2p}
\end{equation}
\end{proof}

\begin{remark}\label{r:even-even}
$X$ was approximated by cutting off higher order terms. If $q=2p$ the corresponding curve becomes  a non-reduced parabola, or a part of it.  In case $ p= even$ a part of parabola is double covered. The same happens with the caustic of the approximation. 
We perturb the example slightly to discover the effect of higher order terms (e.g.):
$$ x = t^p  \; , y=  c t^{2p} + t^{2p+1}.$$
and it follows (after a computation), that also the new parametrization gives a caustic with  a (2,3) cusp.

\end{remark}

\begin{cor}
{\bf  From the proof of Theorem \ref{t:NinD} about caustics:}
\begin{itemize}
\item If $\rho < 1$: The caustic get a cusp of type $(2p-q, p)$ from a cusp of type $(p,q)$ of $X$ at $u=x$,
\item  $\rho = 1$:  We get a   $(2,3) $ cusp at $u \ne x$,
\item  $\rho >1$:   The limit of the caustic is infinity, asymptotic to the normal line $u_1=0$.
\end{itemize}
\end{cor}
Compare also \cite{Zhang-Pei}, where an involute-evolute duality was studied.

\subsection{The discriminant}
\begin{theorem}\label{t:top-discr}
The topological ED-discriminant $\Delta_{tED}$ consists of the caustic (evolute), together with all normal lines at  geometric singular points  and the normal lines at strict $C^1$-points of type $q= odd$. 

\end{theorem}
\begin{proof}
This is a corollary of Theorem \ref{t:NinD} and its proof.
\end{proof}

\subsection{Multiple intersections}
In case several local branches intersect in $x \in X$, there exists  for  each of these branches a well defined focal point (perhaps at infinity). This holds not only for smooth branches, but also for the singular branches. In case several normal lines coincide then there can be several focus points on the same line.

Let $x \in X$ be an intersection point of several local branches, Let $N_x$  be the union of all normal lines $N_{x,\Gamma}$ to these  branches, then:
\textit{If $u \notin N_x$ then $\sq_u$ us topologically regular at $x \in X$. }\\
The effect of including multiple intersections is that $\Delta_{tED}$t will be extended by $N_x$ for every multiple point $x \in X$.
We don't give details.

\subsection{Counting critical points}
After the local studies we now consider global behavior. 
We will  apply (topological) Morse theory. 
We consider for a single branch a topological component of its normalization $\gamma: \mathbb{R} \to \mathbb{R}^2 \; \mbox{\rm or} \; S^1  \to \mathbb{R}^2$. We will assume in the rest of the manuscript, that \textbf{$X$ is  compact and has a single component}, i.e. the source is $S^1$, unless stated differently. We leave the other case to the reader.

In the smooth case we call the interval $[xu]$ the normal from $u$ to $x \in X$ if $x$ a critical point of $\sq_u$. This interval is a subset of the normal line $N_x$. We extend this definition to the case that  $x$ is a geometric singular point of $X$ and use the name s-normal  if $x$ is  a topological critical point of $\sq_u$. This may look a bit strange, since no longer $[xu] \subset N_x$ is valid; but $[xu]$  lies in the normal cone at $x$,  if one extends properly the PL-definition of normal cone to our case. \\

For any $u \notin \Delta_{tED}$ we count the number of critical point of $\sq_u$. Notation $N(u)$. 
In this way we are counting \textit {all normals} through $u$.\\
Critical points $x \in X$ of $\sq_u$ determines normals $xu$ of two types
\begin{itemize}
\item  the (traditional) normal  $xu$  if $x \in X^{reg}$  
\item  s-normals  $xu$ if $x \in X$ is a geometric singular point of $X$.
\end{itemize}

Let $m(u)$ the number of minima ; $M(u)$ the number of maxima. 
Any cuspidal point $x \in X$ contributes one critical point for any $u \notin \Delta_{tED}$. 
Let $m_{\prec}(u)$ the number of minima at cuspidal points ; $M_{\prec}(u)$ the number of maxima at the cuspidal points; $N_{\prec}(u)$ the total number of cuspidal points.
\begin{prop}\label{p:count-normals}
For any $u \notin \Delta_{tED}$  we have:
\begin{itemize}
\item The total number of normals $N(u)  = M(u) + m(u) \ge 2$,
\item Maxima and minima: $m(u)=M(u) \ge 1$ 
\item Each cuspidal point contributes to at least one normal:\\
$ N(u)   \ge M_{\prec}(u) + m_{\prec}(u) + | M_{\prec}(u)-m_{\prec}(u)|  \ge  N_{\prec}(u)$
\end{itemize}
\end{prop}

We next discuss intersections.
\begin{lemma}
Near to an intersection of two smooth branches  $x \in X$ there is a neighborhood such that for any $u$  $\sq_u$ has 2 minima near $x \in X$. As a consequence $N(u) \ge 4$ in that neighborhood.
\end{lemma}
\begin{proof}
At each of the branches $\sq_u$ has a Morse minimum at the intersection point.
\end{proof}

\subsection{Minimal number of critical points}
\begin{prop}
If $\sq_u$ has precisely 2 critical points for almost all $u \in \mathbf{R}^2$  then $X$ is a round circle.
\end{prop}
\begin{proof}
Assume $X$ is not a single round circle.  The smallest circle containing $X$ has finitely many (at least) two points in common with $X$. These correspond to at least two maxima;  so we have at least 4 critical points.
\end{proof}

\subsection{ Maximal and minimal distances}
In several applications one is interested in the maximal or minimal distance from a point $y$ to the object.
It is clear (see the examples) that absolute and relative  maximum and minimum can occur also  at s-normals . These are not considered  via the ED-degree approach in \cite{DHOST}  which deals only with the smooth part of $X$. The ED-degree counts only the traditional normals (including the complex ones)\\

\subsection{Examples}
\begin{Ex}[Global 2-7- cusp]
The pictures of the caustics in Figure \ref{fig:TenCases} wil contain often non-local information. We explain this in the 2-7 cusp-example:
$ \;  \; x_1(t) =  t^2 \; \; ; \; \; x_2(t) = t^7 $.

The caustic avoids the singularity $x$ at the origin.
The normal line $\{u_1=0\}$ in $x$ belongs to the $\Delta_{ tED}$, but is no part of the caustic. 
The  caustic cuts the normal; but be aware: this is not a focal point! In the 1-1-correspondence between evolute and curve the origin with $t=0$ corresponds with infinity. Only the asymptotic parts at infinity are the part of the caustic (very) near to $x$. Starting from $t=0$ in positive $t$-direction, the caustic   follows the normal line at the right side  to a cusp point.  This point corresponds  to a vertex point away from $x$ on $X$; the intersection point corresponds to some other point on $X$. 

\begin{center}
\begin{figure}[h]
\includegraphics[width=5cm]{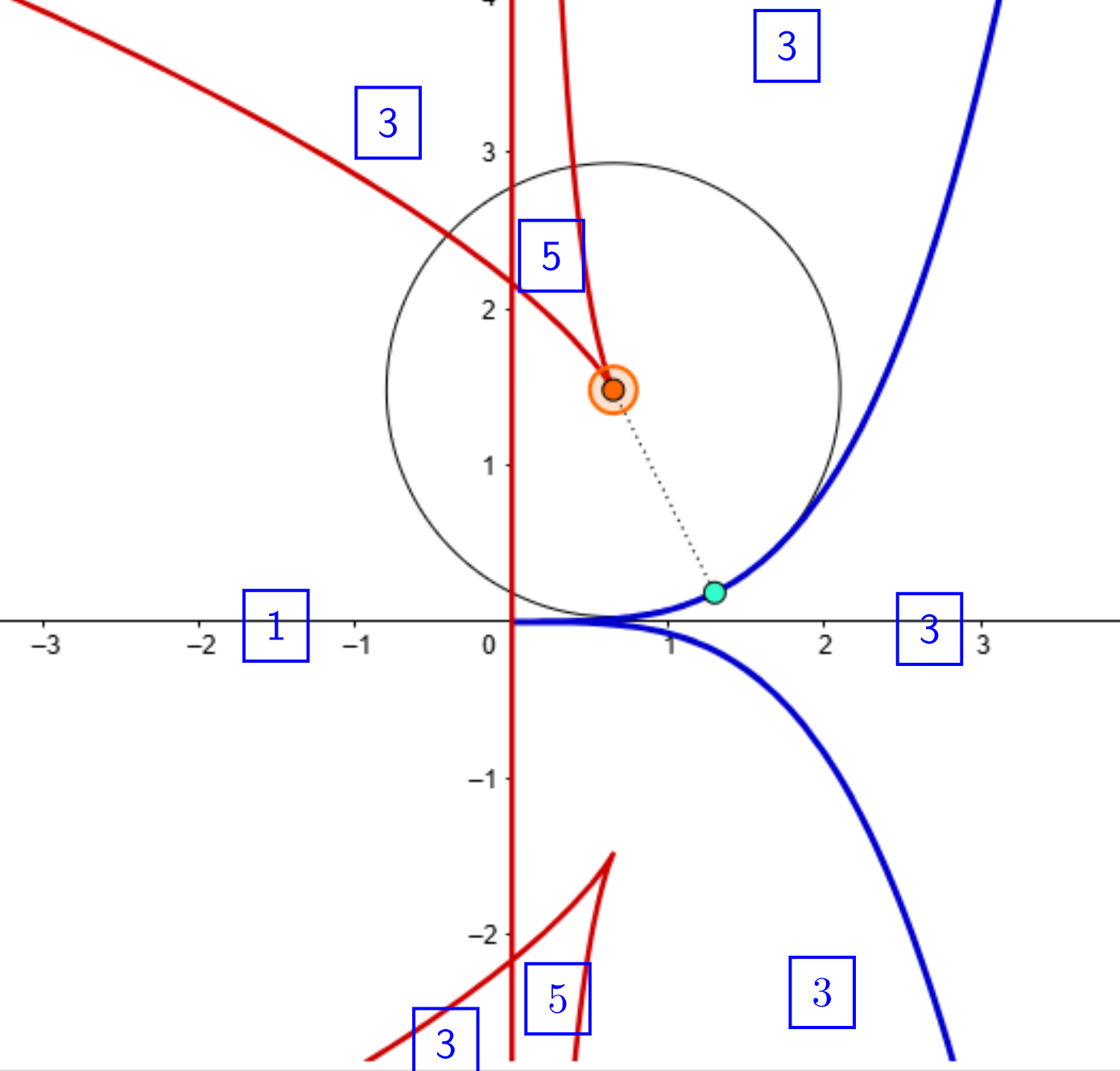}
\caption{The 2-7 cusp; the number of normals is indicated by a boxed number.}
\end{figure}
\end{center}

\end{Ex}
\begin{Ex}[Drop]
Let $X$ be given by:
$$ x_1(t) = \sin(t)   (\cos (t) - 1) \; \; ; \; \; x_2(t) = \cos (t)  $$
In Figure \ref{fig:drop}  The tED-discriminant is drawn together with the number of critical points in its complement. There is one distinguished normal: at the cusp of $X$.

\begin{center}
\begin{figure}[h]
\includegraphics[width=5cm]{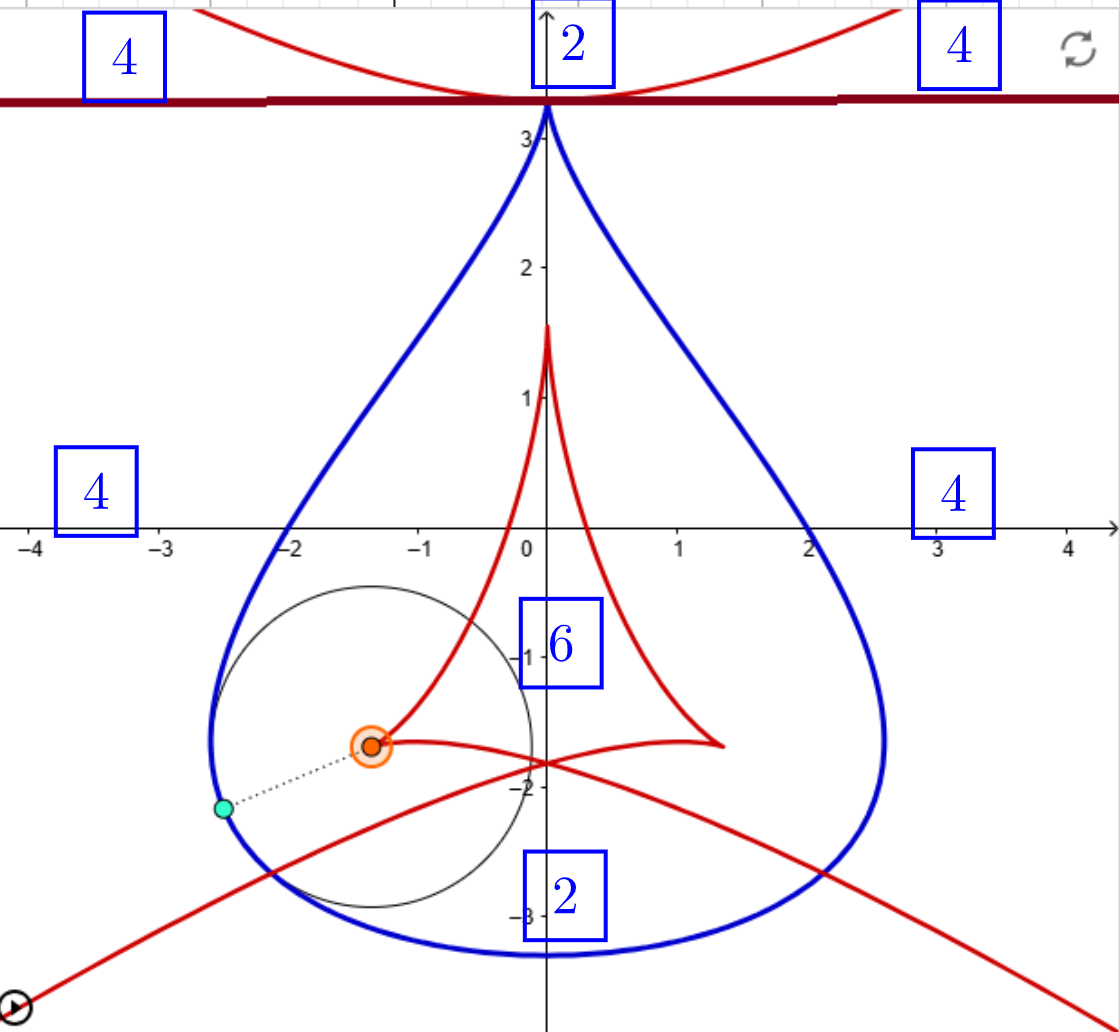}\label{fig:drop}
\caption{Drop and its discriminant; the number of critical points of $\sq_u$ is indicated by a boxed number. }

\end{figure}
\end{center}
\end{Ex}

\section{Double normals and singularities}\label{s:double}

\subsection{Double normals}

Double normals have been studied in the context of convex bodies, e.g.  in relation to several types of affine diameters, cf \cite{So} for a survey. Kuiper \cite{Kuiper}, answering a question of Klee, showed  that any convex body in $\mathbb{R}^n$ has at least $n$ double normals.

 Takens-White \cite{TW} studied double normals for smooth immersed submanifolds. They obtained bounds for the number of double normals in terms of Betti-numbers.

In this section we restrict to $\mathbb{R}^2$, where we want to mention also the work of Halpern \cite{Halpern1}. We will first study double normals between two different curves and after that on a single curve.

\subsection{Double normals of two different curves}

We start with two curves $X_a$ and $X_b$, both  algebraic with singularities,  given via parametrizations $a(t)$ and $b(u)$.  Consider the squared distance between points on $X_a$ and $X_b$:
$$\mathbb{B}  : X_a \times X_b \rightarrow  \mathbb{R}, \ \ \ \  \mathbb{B}(t,u):=|a(t)-b(u)|^2.$$

We assume that the intersection $X_a \cap X_b$ contains only finitely many points or is void. By genericity the critical points $(a,b)$ of $\mathbb{B}$ satisfy $a \notin \Delta_b$ and $b \notin \Delta_a$.

\begin{definition}
The function $\mathbb{B}$ is called topologically regular (top-regular) at $(a,b)$ if $\mathbb{B}$ is topologically equivalent to a linear function; topologically critical (top-critical) if it is not top-regular.
\end{definition}

\begin{prop}\label{p:double-crit}
The top-critical points of $ \mathbb{B}$ are pairs $(a,b)$ of one of the following types:
\begin{itemize}
\item rr-normals (traditional double normals),  i.e. intervals $ab$ such that $ab$ is a normal in $a$ and $ba$ is  normal in $b$. Both $a$ and $b$ regular points of resp. $X_a$ and $X_b$,
\item rs-normal:   one end point is a geometric singular point of one curve;
 $ab$ is normal at the other point, which is assumed to be  regular,
 \item ss-normals: intervals between two geometric singular points of $X_a$ and $X_b$,
 \item the intersection points $X_a \cap X_b$.
\end{itemize}
\end{prop}
\begin{proof}
We work on the normalizations of the curves and use parametrizations of type (\ref{e:Puis}) at singular points. 
We choose the coordinates such that
$$ a(t) = (-1,0) +\varrho_{\alpha}(t^p, c_a t^q)   \;  \mbox{\rm with} \;  p <q  \;  \mbox{\rm and}\;  c_a = c_a[t] \; ; c_a[0] \ne 0  ;$$
$$ b(u) =\;  \; (1,0) + \varrho_{\beta} (u^r, c_b u^s)  \;  \mbox{\rm with} \;  r <s  \;  \mbox{\rm and}\;  c_b = c_b[u]\; ; c_b[0] \ne 0. $$
where $\varrho_{\alpha}$ is a rotation over $\alpha$. We have:
$$\mathbb{B} (t,u) =
-2 c_a c_b t^q u^s \cos (\alpha -\beta )-2 c_a t^q u^r \sin (\alpha -\beta )+c_a^2 t^{2 q}$$ $$-4 c_a \sin (\alpha ) t^q +2 c_b t^p u^s \sin (\alpha -\beta )+c_b^2 u^{2 s}+4 c_b \sin (\beta ) u^s$$ $$ - 2 t^p u^r \cos (\alpha -\beta )+t^{2 p} -4 \cos (\alpha ) t^p +u^{2 r}+4 \cos (\beta ) u^r+4.
$$
The dominant terms are:
\begin{equation}\label{dominant}
-4 \cos (\alpha ) t^p +4 \cos (\beta ) u^r.
 \end{equation}

There are two possible degenerations:
\begin{itemize}
\item $\cos (\alpha )= 0$, i.e.  $b$ lies on the normal to the curve $a(t)$ in $a$,
\item $\cos (\beta) = 0$, i.e. $a$ lies on the normal to the curve $b(u)$ in $b$.
\end{itemize}

Assume now: No degeneration:  the topological type of $\mathbb{B}$  at $(a,b)$ is determined by the dominant part.
We draw the following conclusions:
\begin{itemize}
\item If $p = odd$, or $r= odd$ then $\mathbb{B}$ is topologicaly regular (top-regular) at $(a,b)$,
\item If $p = even$ and $r = even$,
 then $\mathbb{B}$ behaves as  the Morse type $ -4 \cos (\alpha ) t^2 +4 \cos (\beta ) u^2$.
 It will have minimum, maximum or saddle, depending on the values of $\alpha$ and $\beta$.
\end{itemize}

{\bf Case rr:} Let us assume now that both curves are at least $C^1$-smooth. So we have $p=odd$ and $r=odd$.
If no degenerations: we have topological regularity. We are left with $\cos (\alpha) = 0 \;  \&  \;\cos (\beta) = 0$.
Which means that $ab$ is a double normal (in the usual sense). The dominant terms are now:
\begin{equation}\label{rr-condition}
 -4 c_a (-1)^k t^q +t^{2 p} +4 c_b (-1)^l u^s+u^{2 r}.
 \end{equation}
Note $  \alpha = \frac{\pi}{2} + k \pi$ and $\beta = \frac{\pi}{2} + l \pi$. 

We can distinguish between  the two independent cases:\\
$q < 2p \; ; \; q > 2p \; ; \; q = 2p$,\\
$s < 2r \; ; \; s > 2r \; ; \; s = 2r$.\\
This will result in: top-regular, maxima, saddle and minima for $\mathbb{B}$ . We leave this to the reader.
Note that in case $q=2p$ and $s=2r$ we get
\begin{equation}\label{rr-degenerate}
  -4 c_a (-1)^k t^{2p} +t^{2 p} +4 c_b (-1)^l u^{2r}+u^{2 r}.
\end{equation}
and the type depends on the positions of the two focal points. 
NB. This is the $C^2$-case and will be treated in terms of differential geometry and standard Morse theory below.

{\bf Case rs: } Lets us assume $a$ is a  smooth point and $b$ geometrically singular. We look again to the dominant terms (\ref{dominant}). Now $p = odd$ $r= even$: top-regular as long as $\cos(\alpha) \ne 0$.  Degenerate iff $ba$ is normal at $a$. In that case
the dominant and some nearest higher order terms are:
\begin{equation}\label{rs-condition}
  -4 c_a (-1)^k t^q +t^{2 p} +4 \cos (\beta ) u^r   +4 c_b \sin (\beta ) u^s+u^{2 r}
\end{equation}
Note $  \alpha = \frac{\pi}{2} + k \pi.$ 

As soon as $\cos (\beta) \ne 0$ The dominant terms are:
$$  -4 c_a (-1)^k t^q +t^{2 p} +4 \cos (\beta ) u^r.  $$
We can distinguish between
$q < 2p \; ; \; q > 2p \; ; \; q = 2p$:
\begin{itemize}
\item $q < 2p$: dominant: $-4 c_a (-1)^k  t^q +4  \cos(\beta ) u^r$ ; if $q$ is odd: top-regular, else minimum, maximum or saddle,
\item $q > 2p$: dominant: $ t^{2 p} +4 \cos (\beta ) u^r  $ \\
Since $r$ is even: minimum or saddle,
\item $q =2p$: dominant $  -4 c_a (-1)^k t^{2p} +t^{2 p} +4 \cos (\beta ) u^r  $.
 Since $r$ is  even: minimum, saddle or maximum depending on the  place of the focal points.
\end{itemize}

{\bf Case ss:}   
 Lets us assume that  $a$  and $b$ are both geometrically singular. We look again to the dominant terms (\ref{dominant}). Now $p = even$ $r=even$. \textit{This is the generic case: The critical point is equivalent to a Morse critical point} as long as $\cos(\alpha) \ne 0$  and  $\cos(\beta) \ne 0$. Degenerate critical point of $\mathbb{B}$ iff $ba$ is normal at $a$ or $ab$ is normal at $b$. Assume the first, we can use again (\ref{rr-degenerate}):
dominant and nearest higher order terms are:
\begin{equation}
  -4 c_a (-1)^k t^q +t^{2 p} +4 \cos (\beta ) u^r +4 c_b \sin (\beta ) u^s+u^{2 r}.
\end{equation}
Note $  \alpha = \frac{\pi}{2} + k \pi.$ 
There is no difference with the  rs-case (\ref{rs-condition}).

If degenerate in both points $a$ and $b$ we get similar to the rr-case (\ref{rr-degenerate}):
$$  -4 c_a (-1)^k t^q +t^{2 p} +4 c_b (-1)^l u^s+u^{2 r}.$$
Note $  \alpha = \frac{\pi}{2} + k \pi$ and $\beta = \frac{\pi}{2} + l \pi$. 

We can distinguish between  the two independent cases:\\
$q < 2p \; ; \; q > 2p \; ; \; q = 2p$,\\
$s < 2r \; ; \; s > 2r \; ; \; s = 2r$.\\
This will result in: top-regular, maxima, saddle and minima . We leave this to the reader.
Note that in case $q=2p$ and $s=2r$ we get
\begin{equation}
 -4 c_a (-1)^k t^{2p} +t^{2 p} +4 c_b (-1)^l u^{2r}+u^{2 r}
\end{equation}
and the type depends on the positions of the two focal points.

{\bf Case intersection}: If the intersection of $X_a$ and $X_b$ is non-void then we have clearly an absolute minimum at each intersection point.
\end{proof}

In the smooth $C^2$ case we can use elementary differential geometry and Morse theory to determine the index of the critical points of $\mathbb{B}$.
Assume that $X$ is parametrized by arc length. The Frenet-frame consists of the pair $(\mathbf{t},\mathbf{n})$ (tangent vector, normal). We recall the  Frenet-Serret formula's $\mathbf{t}'= \kappa \mathbf{n}$, $\mathbf{n}'= - \kappa \mathbf{t} $ . \\ 

For a critical point $x \in X$ of  $\sq_y$ we  define $I(x) = \kappa  (x-y)\cdot \mathbf{n} +1$, the \textit{index indicator}. It is the second derivative of  $\sq_y$ at $x \in X$. If $I(x) < 0$ then there is a maximum; when $I(x) > 0$ a minimum.\\

Next consider frames in two points $a \in X_a$ and $b \in X_b$. We will add $a$ and $b$ as lower-index to the vector.
 Let $\gamma$ be the angle between $\mathbf{t}_a$ and $\mathbf{t}_b$.

\begin{prop}
{\bf Index statement.}
Let $ab$ be a double normal (in two smooth $C^2$ points) which is Morse (as critical point of $\mathbb{B}$) with Hessian
$$ H(a,b) =  I(a) I(b) - 1. $$ The  double normal has the following type:
\begin{enumerate}
\item saddle type if  $H(a,b) < 0$,
\item maximum type if $H(a,b) > 0$ and  $I(a)< 0$ and $I(b) < 0 $,
\item minimum type if $H(a,b) > 0$ and  $I_(a) > 0$ and $I(b) > 0 $.
\end{enumerate}
\end{prop}
\begin{proof}
 We will compute the first and second order partial derivatives of $\mathbb{B}$ in non-boundary points.
 $$ B_1 = (a-b) \cdot  a'\; \; ; \; \; B_2 =(b-a )\cdot b'   $$ 
$$ B_{11} = \kappa_a (a-b )\cdot \mathbf{n}_a +\mathbf{ t}_a^2 = \kappa_a (a-b) \cdot \mathbf{n}_a  +1 \; \; \mbox{\rm at} (a,b)  $$
$$ B_{22} = \kappa_b (a - b )\cdot \mathbf{n}_b +\mathbf{ t}_b^2 = \kappa_b(a-b) \cdot \mathbf{n}_b  +1 \; \; \mbox{\rm at} (a,b) $$
$$ B_{12}=B_{21} = - \mathbf{t}_a \cdot \mathbf{t_b}= - \cos^2 \gamma = -1.$$
Apply next the standard criteria for Morse types of a function of two variables.
\end{proof}

Is there a relation between the type of of double normal  and the types of the two associated `single' normals? \\
It follows immediately that if $ab$ is of minimum or maximum type then then normal $ab$   is of the same type as the (`single') normal from $a$ and also as normal from $b$.  If the double normal is of saddle type there is no general rule.
The types depend on the positions of the focal points of $X_a$ and $X_b$ with respect to $a$,$b$.\\
We summarize:

\begin{itemize}
\item   Double normal maximal     gives 2 maxima at the end points,
\item  Double normal minimal     gives two minima at the end points,
\item Double normal saddle       gives Max-min or Max-Max or min-min at the end points.
\end{itemize}

\subsection{Maximal and minimal length of double normals}.\\
We will only work with compact spaces $X_a$ and $X_b$, i.e both are images of $S^1$.
In case of non-void  intersections $X_a \cap X_b$ the the minimum of $ \mathbb{B}$ is zero.
In the other cases the absolute minimum and absolute maximum of $\mathbb{B}$ can be attained in smooth points $(a,b)$ or in one or both singular points.
These two extreme values are important in computational geometry.

\subsection{Morse theory for double normals}\

We will use Morse theory on the smooth normalization of $\tilde{X}_a$ and $\tilde{X}_b$. Due to our compactness assumptions  each topological component is homeomorphic to $S^1$. Let $M= \tilde{X_a} \times \tilde{X}_b$. Note that on intersection points $X_a \cap X_b$ the function $\mathbb{B}$ has an absolute mininimum zero. In the counts below we will include this in the number of double normals of minimum type. The Morse (in)equalities imply:

\begin{prop}\label{p:double2-count}

 On $M$
   we have Morse count for $\mathbb{B}$
   \begin{itemize}
   \item
  $\sharp(Max)-\sharp(Saddles) +\sharp(Min)  =0.$
 \item 
    The number of double normals is at least $ 4$.
    \end{itemize}
\end{prop}

\begin{prop}\label{p:extra-double}
For  a generic pair of curves:
\begin{itemize}
\item Every   cuspidal point  contributes to at least two double normals,
\item Two cuspidal points at different curves contribute to a double normal between the cusps,
\item An intersection point between $X_a$ and $X_b$ contributes to a (extra) minimum. As a consequence there is also an extra saddle.
\end{itemize}
\end{prop}
\begin{proof}
This follows from minimum-maximum arguments, Proposition \ref{p:double-crit} and Morse inequalities.

\end{proof}

\subsubsection{Drop-Ellipse normals}
We consider two curves (Figure \ref{fig:dropellipse}):
\begin{itemize}
\item The drop:   $a(t)=(\,\sin t(\cos t -1), \,  \cos t)$
\item The ellipse:   $b(t) = (\,A \sin t , \,B \cos t)$ with $A=3, B=2$.
\end{itemize}
The ellipse is smooth, the drop has a singular point $(0,1)$ at $t=0$.
\begin{center}
\begin{figure}
\includegraphics[width=6cm]{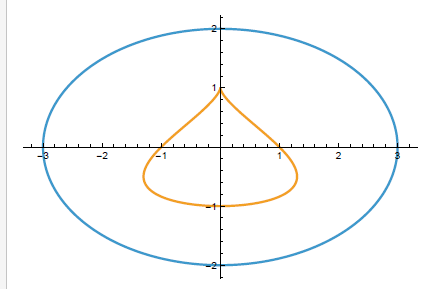}
\caption{Drop inside ellipse }\label{fig:dropellipse}
\end{figure}
\end{center}
The squared Euclidean distance function is:
$$\mathbb{B}(s,t) = (\sin t (\cos t -1)- A \sin s)^2 + (\cos t - B \cos s)^2$$
This function is defined on a torus and due to the parametrization it is a differential function. The parametrization of $a$ is defined on a smooth normalization.
Computations with Mathematica give us 14 critical points. 
We can use the notation $3m + 7s +4 M$; See the Figure \ref{fig:functionB}.

\begin{figure}[ht]
\begin{minipage}[b]{0.45\linewidth}
\centering
\includegraphics[width=\textwidth]{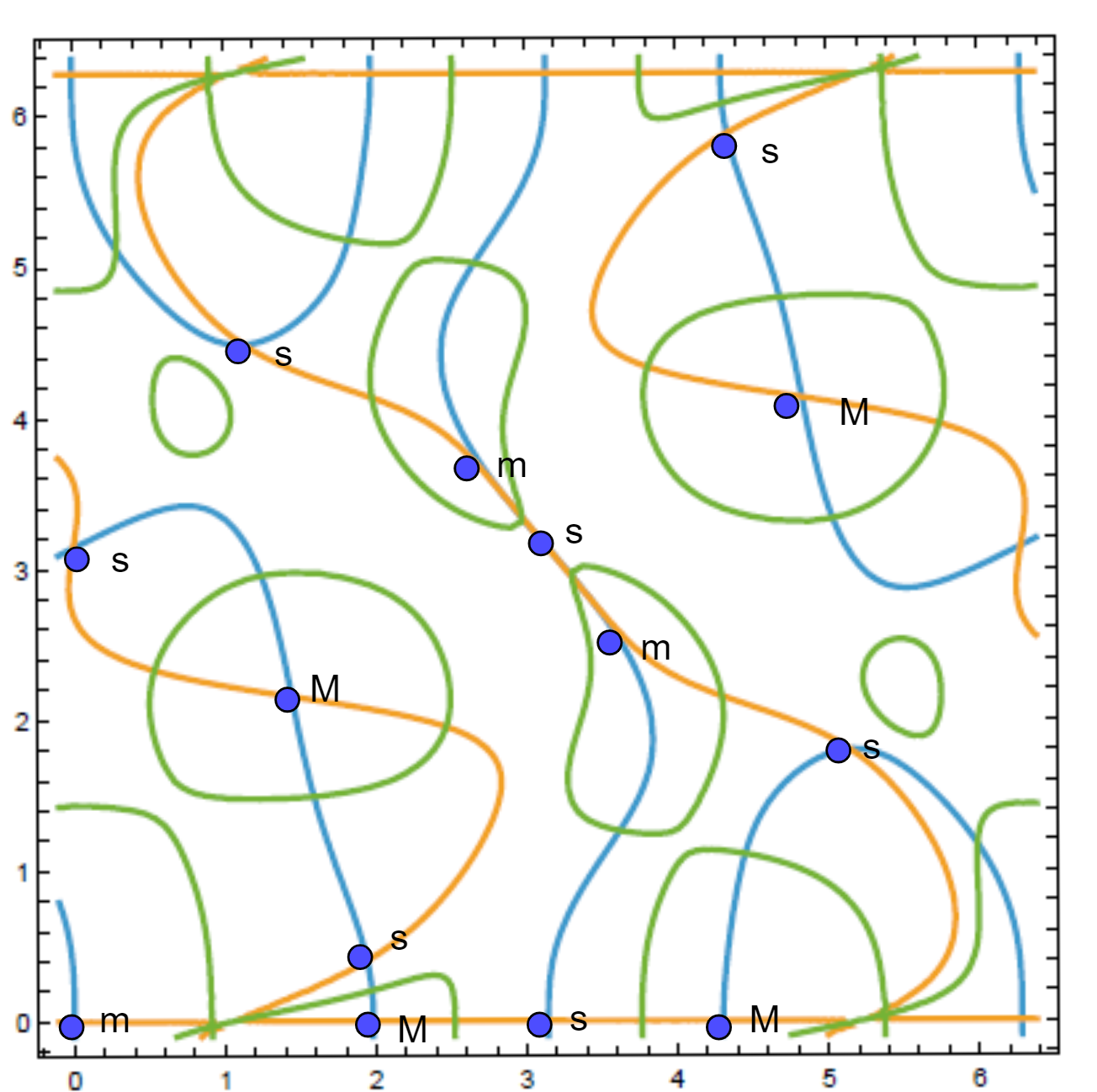}
\label{fig:image1}
\end{minipage}
\hspace{0.05\linewidth}
\begin{minipage}[b]{0.45\linewidth}
\centering
\includegraphics[width=1.14\textwidth]{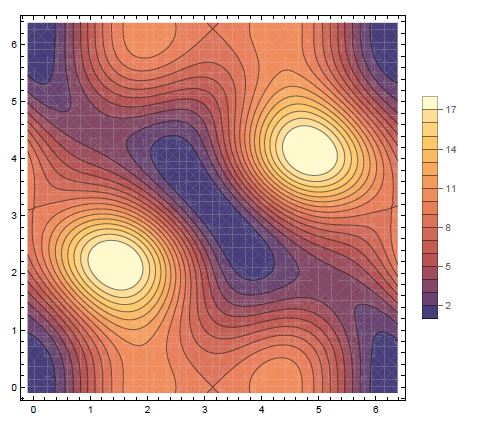}
\label{fig:image2}
\end{minipage}
\caption{The function $\mathbb{B}$:\\  Left: {Partial derivatives equal zero in blue and yellow; Hessian equal zero in green} ; Right: Level curves}\label{fig:functionB}
\end{figure}

\subsection{Double normals on a single curve}

Let $X = X_a = X_b$.  We still assume $X$  bounded and a single topological component.
Note that $\mathbb{B}$ has  a non-isolated singularity on the diagonal  $\Delta X$ with the absolute minimum value $0$. 
The function $\mathbb{B}$ is is symmetric with respect to the permutation of $x_1$ and $x_2$.  Therefore we consider the quotient space. Consider the set of unordered pairs of points of the normalization of the curve $X$. $$\mathcal{N}:=\tilde{X}\times \tilde{X}/(x_1,x_2)\sim (x_2,x_1).$$
The function $\mathbb{B}$  induces the function:
$$\mathbb{B}: \mathcal{N} \rightarrow  \mathbb{R}, \ \ \ \ \mathbb{B}(x_1,x_2):= |x_1-x_2|^2.$$

We count double normals as critical points of $\mathbb{B}$. Self intersections correspond to  pairs  with different points on the normalization. They are (absolute) minima of $\mathbb{B}$ and are included in the counting of minima below.

\begin{prop}\label{p:double-count}
 For a generic curve, we have Morse count: 
  \begin{itemize}
  \item The number of double normals is at least $2$.,
 \item  $\sharp(Max)+\sharp(Min) = \sharp(Saddles).$
 \end{itemize}
\end{prop}
\begin{proof}
  $\mathbb{B}$  is defined on the Moebius band, and attains its global minimum $0$ on the boundary. Contraction of the boundary yields the projective plane. So now $\mathbb{B}$ is defined on the projective plane. All the critical points correspond bijectively to double normals except for the  boundary.
\end{proof}

\begin{remark}
It follows from the Morse inequalities that there must be at least one saddle point and one maximum. It looks perhaps surprising, that there is no guarantee  for a double normal of minimum type. 
\end{remark}

Proposition \ref{p:extra-double} implies for a single curve:
\begin{prop}\label{p:extra-selfdouble}
For  a generic curve:
\begin{itemize}
\item Every   cuspidal point  contributes to at least one double normal,
\item Two cuspidal points contribute to a double normal between the cusps,
\item An intersection point between $X_a$ and $X_b$ contributes to a (extra) minimum. As a consequence there is also an extra saddle.
\end{itemize}
\end{prop}

\begin{remark} For the statement about intersection points, see  Halpern \cite{Halpern1}, who proved several counting formula's  for smooth immersed closed curves in the plane. This concerns not only double normals but also double tangens and normal-tangent coincidence. 
\end{remark}


\subsection{Examples}

\begin{Ex} Hypercycloids are classical curves: the locus of a point on a circle rolling on a bigger circle.  See Figure \ref{fig:dn-pics} Left. We consider  a hypercycloid with 3 cusps:\\
$x(t) = 2 \cos (t) + \sin (2t) \; ; \; y(t) = 2 \sin (t) + \cos (2t)$\\
There are 6 normals: 3 ss-normals of maximum type;  3 rs-normals of saddle type. Computations can be easily done with computer software, such as Mathematica.

\end{Ex}

\begin{Ex}
Lemniscate with 2 cusps:\\
$x(t) = \frac{\cos(t)}{1 + \sin^2(t)}   \; \; \; ; \; \; \;
y(t)= \frac{\sin^3 (t)\cos(t)}{1 + \sin^2(t)}$\\
There are 3 rr-normals: 1 ss-normal of maximum type;  2 normals of saddle type.  Moreover there is 1 minimum, corresponding to te intersection point. See Figure \ref{fig:dn-pics} Right. Compare  also the Bernoulli lemniscate.

\begin{figure}[ht] \label{fig:dn-pics}
\begin{minipage}[b]{0.45\linewidth}
\centering
\includegraphics[width=\textwidth]{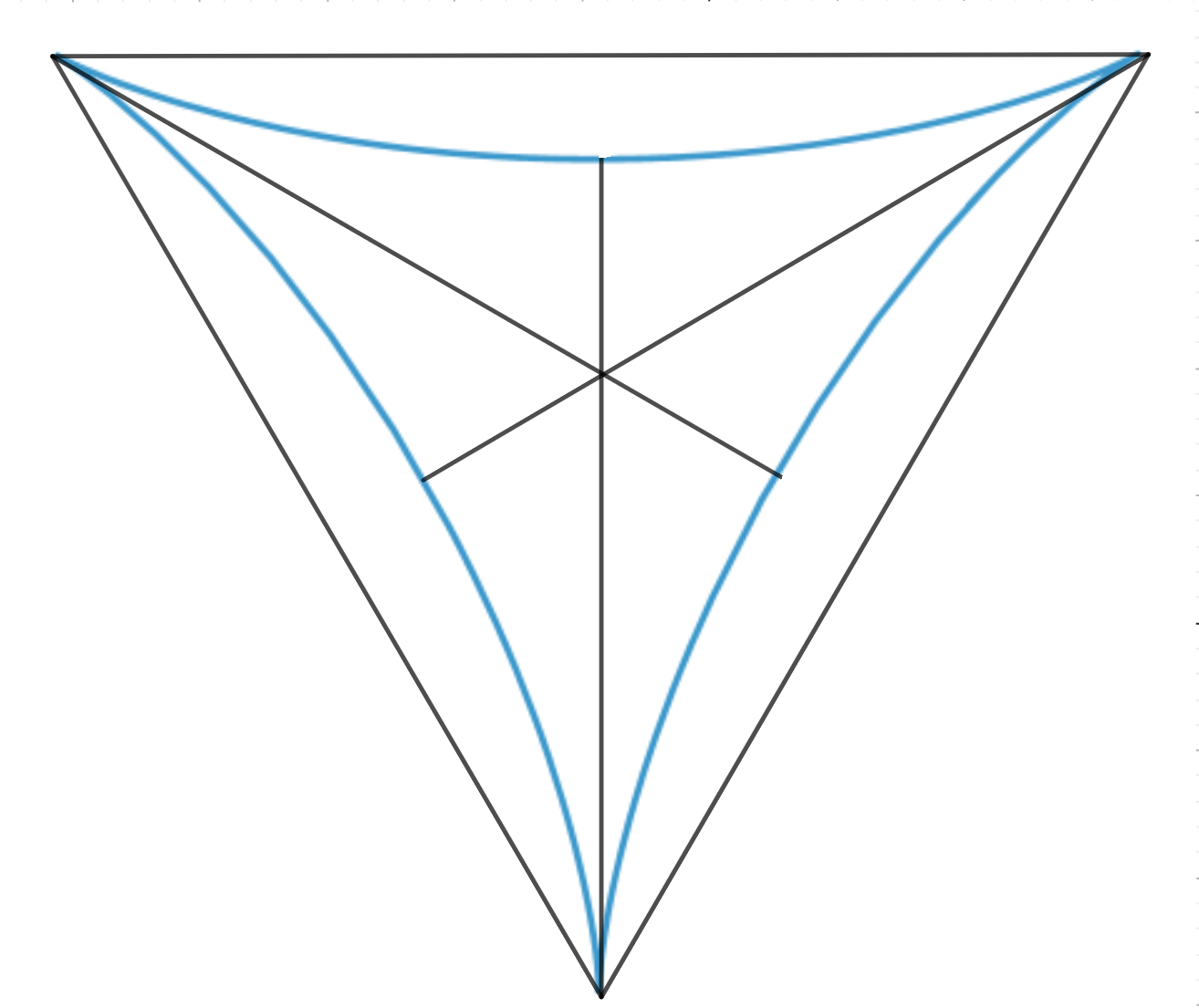}
\label{fig:image1}
hypercycloid
\end{minipage}
\hspace{0.05\linewidth}
\begin{minipage}[b]{0.45\linewidth}
\centering
\includegraphics[width=\textwidth]{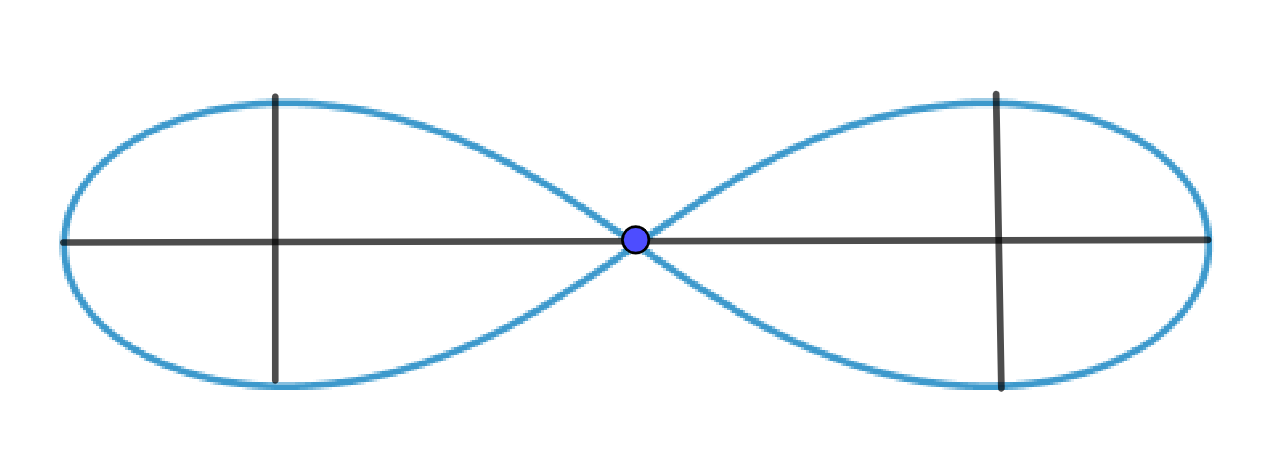}
Bernouilli  lemniscate \\
\vspace{0.2cm}
\includegraphics[width=\textwidth]{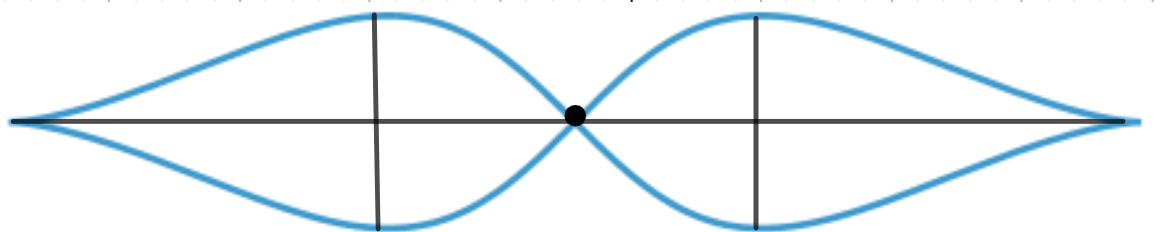}
lemniscate with 2 cusps

\vspace{0.2cm}
\label{fig:image2}
\end{minipage}
\caption{Double normals for three curves}
\end{figure}

\end{Ex}
\section{The complex case}

The paper \cite{JST} contains a detailed analysis of the ED-discriminant of a algebraic curve in the complex plane $\mathbb{C}^2$.
The  distance function is the holomorphic function 
\begin{equation}\label{e:sqdc}
\sq_u:X\rightarrow \mathbb{C}, \ \ \  \sq_u(z):=(z-u)^2,
\end{equation}
which in contrast with the real case  (\ref{e:sqd0}) does not satisfy the usual distance properties. The definition of discriminant is adapted to the complex case. Algebraic singular points of $X$ are all geometrically singular (i.e. never $C^1$-smooth).

It is shown, that the complex ED-discriminant can be decomposed in the following way:
  $$ \Delta_{\tED} = \Delta^{focal} \cup \Delta^{iflex}  \cup \Delta^{\atyp} \cup \Delta^{\sing}  $$
  
  $(1).$    \emph{The focal set} $\Delta^{\focal}$ represents the tradional evolute and is birational with the regular part of $X$. It is generically curved (may contain components of higher degree); while the other components are lines and related to certain isolated points on $X$ (or its completion  at infinity).

 The line components in the decomposition are:
 
 $(2).$  \emph{The atypical discriminant} $\Delta^{\atyp}$, due to the Morse points which are ``lost'' at infinity. A necessary condition is that  the completion $\overline{X}$ passes through an isotropic point.

 $(3).$  \emph{The singular  discriminant} $\Delta^{\sing}$, due to the Morse points which move to singularities of $X$. 
 
  $(4)$ \emph{The iflex discriminant } $\Delta^{\iflex}$, due to collision of Morse points on $X_{\reg}$, which move to flex points on $X$, with isotropic tangent lines (for short:\emph{ iflex points}). 
  
  Since the atypical and iflex discriminant don't occur in the real case, we have there only
   $$ \Delta_{tED} = \Delta^{focal} \cup   \Delta^{\sing},  $$
 where we note, that  $\Delta^{\sing}$ in the real case can contain less normal lines than its complex companion.


\end{document}


\end{document}